%% file: Master.tex
\documentclass[a4paper,oneside,italian,english,british]{amsart}
\usepackage[T1]{fontenc}
\usepackage[latin9]{inputenc}
\usepackage{color}
\usepackage{babel}
\usepackage{prettyref}
\usepackage{amsthm}
\usepackage{subscript}
\usepackage[unicode=true,pdfusetitle,
 bookmarks=true,bookmarksnumbered=true,bookmarksopen=false,
 breaklinks=false,pdfborder={0 0 0},backref=false,colorlinks=true]
 {hyperref}

\makeatletter

\pdfpageheight\paperheight
\pdfpagewidth\paperwidth

\numberwithin{equation}{section}
\numberwithin{figure}{section}
\numberwithin{table}{section}
\theoremstyle{plain}
\newtheorem{thm}{\protect\theoremname}[section]
  \theoremstyle{definition}
  \newtheorem{defn}[thm]{\protect\definitionname}
  \theoremstyle{plain}
  \newtheorem{prop}[thm]{\protect\propositionname}
  \theoremstyle{plain}
  \newtheorem{cor}[thm]{\protect\corollaryname}
  \theoremstyle{remark}
  \newtheorem{rem}[thm]{\protect\remarkname}

\newrefformat{prop}{\protect\propositionname\ \ref{#1}}
\newrefformat{fact}{\protect\factname\ \ref{#1}}
\newrefformat{cor}{\protect\corollaryname\ \ref{#1}}
\newrefformat{lem}{\protect\lemmaname\ \ref{#1}}
\newrefformat{rmk}{\protect\remarkname\ \ref{#1}}

\makeatother

  \addto\captionsbritish{\renewcommand{\corollaryname}{Corollary}}
  \addto\captionsbritish{\renewcommand{\definitionname}{Definition}}
  \addto\captionsbritish{\renewcommand{\propositionname}{Proposition}}
  \addto\captionsbritish{\renewcommand{\remarkname}{Remark}}
  \addto\captionsbritish{\renewcommand{\theoremname}{Theorem}}
  \addto\captionsenglish{\renewcommand{\corollaryname}{Corollary}}
  \addto\captionsenglish{\renewcommand{\definitionname}{Definition}}
  \addto\captionsenglish{\renewcommand{\propositionname}{Proposition}}
  \addto\captionsenglish{\renewcommand{\remarkname}{Remark}}
  \addto\captionsenglish{\renewcommand{\theoremname}{Theorem}}
  \addto\captionsitalian{\renewcommand{\corollaryname}{Corollario}}
  \addto\captionsitalian{\renewcommand{\definitionname}{Definizione}}
  \addto\captionsitalian{\renewcommand{\propositionname}{Proposizione}}
  \addto\captionsitalian{\renewcommand{\remarkname}{Osservazione}}
  \addto\captionsitalian{\renewcommand{\theoremname}{Teorema}}
  \providecommand{\corollaryname}{Corollary}
  \providecommand{\definitionname}{Definition}
  \providecommand{\propositionname}{Proposition}
  \providecommand{\remarkname}{Remark}
\providecommand{\theoremname}{Theorem}

\begin{document}
\input{Macros.tex}

\title{A pseudoexponential-like structure on the algebraic numbers}

\author{Vincenzo Mantova}

\date{7th February 2014}
\begin{abstract}
Pseudoexponential fields are exponential fields similar to complex
exponentiation satisfying the Schanuel Property, which is the abstract
statement of Schanuel's Conjecture, and an adapted form of existential
closure.

Here we show that if we remove the Schanuel Property and just care
about existential closure, it is possible to create several existentially
closed exponential functions on the algebraic numbers that still have
similarities with complex exponentiation. The main difficulties are
related to the arithmetic of algebraic numbers, and they can be overcome
with known results about specialisations of multiplicatively independent
functions on algebraic varieties.
\end{abstract}
\maketitle
\input{Introduction.tex}

\input{Construction.tex}

\input{GoodPoints.tex}

\bibliographystyle{plainurl}
\bibliography{Bibliography}

\end{document}

%% file: Macros.tex
\selectlanguage{english}%
\global\long\def\C{\mathbb{C}}

\global\long\def\G{\mathbb{G}}

\global\long\def\N{\mathbb{N}}

\global\long\def\P{\mathbb{P}}

\global\long\def\Q{\mathbb{Q}}

\global\long\def\oQ{\overline{\mathbb{Q}}}

\global\long\def\Z{\mathbb{Z}}

\global\long\def\Cc{\mathcal{C}}

\global\long\def\Lc{\mathcal{L}}

\global\long\def\dom{\operatorname{dom}}

\global\long\def\Gal{\operatorname{Gal}}

\global\long\def\img{\operatorname{img}}

\global\long\def\ld{\operatorname{lin.d.}}

\global\long\def\td{\operatorname{tr.deg.}}
\selectlanguage{british}%

%% file: Introduction.tex
\section{Introduction}

Pseudoexponentiation is a structure introduced by Zilber in \cite{Zilber2005}
in order to find out how $\C_{\exp}$ should look like if it were
well-behaved, at least for the criteria of a model theorist. The unavoidable
problem of $\C_{\exp}$ is that it defines the ring of integers, hence
Peano's arithmetic, defying the model-theoretic tools widely used
in the last decades.

However, Zilber proved that if $\C_{\exp}$ satisfies certain algebraic
conjectures, Peano's arithmetic is in some sense the only true problem.
He showed that there is a sentence $\Psi$, in the infinitary language
$\Lc_{\omega_{1},\omega}(Q)$, which is \emph{uncountably categorical},
and that describes an exponential field which is reasonably similar
to $\C_{\exp}$. Its models have been called pseudoexponential fields,
perfect exponential fields, or Zilber fields. The two main statements
contained in $\Psi$, which are currently only conjectures for $\C_{\exp}$,
are the Schanuel Property and the Strong Exponential-algebraic Closure.

The Schanuel Property is just a rephrasing of Schanuel's Conjecture
for an abstract exponential function $E$, and it asserts that for
any $z_{1},\dots,z_{n}$ in the base field we have
\[
\td(z_{1},\dots,z_{n},E(z_{1}),\dots,E(z_{n}))\geq\ld_{\Q}(z_{1},\dots,z_{n}).
\]

It is well known that the Schanuel Property is not enough to characterise
well an exponential function, as formally shown by Hyttinen in \cite{Hyttinen2005}:
there are $2^{2^{\aleph_{0}}}$ pairwise non-isomorphic surjective
exponential functions on $\C$ satisfying the Schanuel Property and
whose kernel is a cyclic group.

Here we show a related result, in a quite different vein, about the
Exponential-algebraic Closure. We show that if we drop all the assumptions
about transcendence in Zilber's axiom $\Psi$, then we can construct
several model where the Schanuel Property is falsified in the most
drastic way: everything is algebraic!
\begin{thm}
\label{thm:main}There is a function $E:\oQ\to\oQ^{*}$ such that
\begin{enumerate}
\item $E(x+y)=E(x)\cdot E(y)$ for all $x,y\in\oQ$;
\item $\ker(E)=\omega\Z$ for some $\omega\in\oQ^{*}$;
\item $E$ is surjective;
\item for any absolutely free variety $V$ over $\oQ$ there is an $\overline{z}\subset\overline{\Q}$
such that $(\overline{z},E(\overline{z}))\in V$, and the points of
this form are Zariski-dense in $V$.
\end{enumerate}
\end{thm}
If we consider the class of structures $K_{E}$, where $K$ is a field
and $E$ is an exponential function with cyclic kernel, then $\oQ_{E}$
is existentially closed: whenever $\oQ_{E}\subset K_{E'}$, and some
finite system of polynomial exponential equations and inequations
with parameters in $\oQ$ has a solution in $K_{E'}$, then it already
has a solution in $\oQ_{E}$.

The proof of \prettyref{thm:main} is given using an explicit inductive
construction very similar to the one of \cite{Mantova2011a} and it
is described, along with the list of Zilber's axioms, in \prettyref{sec:cons}.
The fact that the construction itself is well-defined, and it works
as desired, is proved in \prettyref{sec:goodpoints}, thanks to some
arithmetic properties of number fields about specialisations that
were analysed in \cite{Masser1989}.

The author would like to thank his supervisor Prof.\ Alessandro Berarducci,
who proposed to study pseudoexponential fields, Jonathan Kirby for
having proposed the problem solved in this paper, and Profs.\ David
Masser and Umberto Zannier for the suggestions about the number-theoretic
part of this paper that greatly simplified the discussion. This work
was part of the author's PhD work at the \foreignlanguage{italian}{Scuola
Normale Superiore} of Pisa, and it has been partially supported by
the PRIN-MIUR 2009 ``\foreignlanguage{italian}{O-mi\-ni\-ma\-li\-tà,
teo\-ria de\-gli in\-sie\-mi, me\-to\-di e mo\-del\-li non\-stan\-dard
e ap\-pli\-ca\-zio\-ni}'', the EC's Seventh Framework Programme
{[}FP7/2007-2013{]} under grant agreement no.~238381, and the FIRB
2010 ``\foreignlanguage{italian}{Nuo\-vi svi\-lup\-pi nel\-la
Teo\-ria dei Mo\-del\-li del\-l'es\-po\-nen\-zia\-zio\-ne}''.

%% file: Construction.tex
\section{The construction}

\label{sec:cons}

\subsection{Zilber's original axiomatisation}

For the sake of clarity, we briefly recall the axiomatisation of actual
pseudoexponential fields. A field $K_{E}$ is a pseudoexponential
field if it satisfies the following list of axioms. The terms in quotation
marks are not defined here; we shall only explain the meaning of the
properties that actually matter for our purposes. We refer the reader
to \cite{Zilber2005,Marker2006} for a more complete treatment of
the subject.

\subsubsection{Trivial properties of $\C_{\exp}$}
\begin{enumerate}
\item [(ACF\textsubscript{0})] $K$ is an algebraic closed field of characteristic
$0$.
\item [(E)] $E$ is a homomorphism $E:(K,+)\to(K^{\times},\cdot)$.
\item [(LOG)] $E$ is surjective (every element has a logarithm).
\item [(STD)] the kernel is a cyclic group, i.e., $\ker E=\omega\Z$ for
some $\omega\in K^{\times}$.
\end{enumerate}

\subsubsection{Axioms conjecturally true on $\C_{\exp}$}
\begin{enumerate}
\item [(SP)] \emph{Schanuel Property}: for every $z_{1},\dots,z_{n}\in K$
\[
\td(z_{1},\dots,z_{n},E(z_{1}),\dots,E(z_{n}))\geq\ld_{\Q}(z_{1},\dots,z_{n}).
\]

\item [(SEC)] \emph{Strong Exponential-algebraic Closure}: for every irreducible
``absolutely free rotund'' algebraic variety $V\subset K^{n}\times(K^{*})^{n}$,
and every finite subset $\overline{c}\subset K$ such that $V$ is
defined over $\overline{c}$, there is a $\overline{z}\in K^{n}$
such that $(\overline{z},E(\overline{z}))$ is a generic point of
$V$ over $\overline{c}$.
\end{enumerate}

\subsubsection{A non-trivial property of $\C_{\exp}$ \cite[Lemma 5.12]{Zilber2005}}
\begin{enumerate}
\item [(CCP)] \emph{Countable Closure Property}: for every irreducible
``absolutely free rotund'' algebraic variety $V\subset\G^{n}$ over
$K$ of ``depth $0$'', and every finite subset $\overline{c}\subset K$
such that $V$ is defined over $\overline{c}$, the set of the points
of $V$ of the form $(\overline{z},E(\overline{z}))$ that are generic
over $\overline{c}$ is at most countable.
\end{enumerate}
For the purposes of this paper, we actually only care about the meaning
of ``absolutely free''. The additional properties ``rotund'' and
``depth $0$'' are deeply related to the presence of the Schanuel
Property, and moreover they have rather complicated definitions, so
we will omit them here.
\begin{defn}
An irreducible algebraic variety $V\subset K^{n}\times(K^{*})^{n}$
is \emph{additively free} over $L\subset K$ if its projection onto
$K^{n}$ is not contained in a proper $\Q$-linear subspace defined
over $L$. In other words, the coordinate functions of the factor
$K^{n}$ restricted to $V$ are $\Q$-linearly independent from $L$.
\end{defn}
We can state a similar property for the multiplicative side $(K^{*})^{n}$.
\begin{defn}
An irreducible algebraic variety $V\subset K^{n}\times(K^{*})^{n}$
is \emph{multiplicatively free} over $M\subset K^{*}$ if its projection
onto $(K^{*})^{n}$ is not contained in a proper algebraic subgroup
of $(K^{*})^{n}$ defined over $M$. In other words, the coordinate
functions of the factor $(K^{*})^{n}$ restricted to $V$ are multiplicatively
independent from $M$.
\end{defn}
Absolute freeness is when we have both properties with $L=K$ and
$M=K^{*}$.
\begin{defn}
A variety $V\subset K^{n}\times(K^{*})^{n}$ is \emph{absolutely additively
free} if it is additively free over $K$.

$V$ is \emph{absolutely multiplicatively free} if it is multiplicatively
free over $K^{*}$.

$V$ is \emph{absolutely free} if it is both absolutely additively
free and absolutely multiplicatively free.
\end{defn}

\subsection{Axioms for $\oQ_{E}$}

Our goal is to build an exponential field $\oQ_{E}$ as similar as
possible to pseudoexponentiation, but clearly without the axiom (SP).
We definitely want, and actually we can, keep the trivial properties
of $\C_{\exp}$ as they are. Moreover, the axiom (CCP) doesn't even
need to be mentioned, as $\oQ$ itself is countable. The only axiom
that requires some changes is (SEC), as it requires the points $(\overline{z},E(\overline{z}))$
to be ``generic'', and in particular of transcendence degree more
than zero, which is not possible in $\oQ$.

The axiom (SEC) is a special form of existential closure adapted to
the presence of (SP) and to Hrushovski's amalgamation: if a system
of equations and inequations in $K_{E}$ has a solution in some ``strong
kernel preserving extension'', than it has already a solution in
$K_{E}$, plus a genericity assumption. For our purpose, we shall
require that if a system of equations has a solution in some kernel
preserving extension of $\oQ_{E}$, then it has a solution in $\oQ_{E}$.
We drop genericity, and we simplify the discussion by also dropping
the word ``strong'', which is due to the presence of (SP) and it
is therefore irrelevant four our purposes.

It can be easily verified that this condition is equivalent to the
following:
\begin{enumerate}
\item [(EC)]For any absolutely free variety $V\subset\oQ^{n}\times(\oQ^{*})^{n}$
there is a $\overline{z}\in\overline{\Q}^{n}$ such that $(\overline{z},E(\overline{z}))\in V$,
and the points of this form are Zariski-dense in $V$.
\end{enumerate}
The two differences with (SEC) are that we do not require $V$ to
be rotund, which is essentially linked to the use of strong extensions
and the presence of (SP), and that we explicitly force the points
$(\overline{z},E(\overline{z}))$ to be Zariski-dense, while in (SEC)
this is automatic by genericity. (As we will note later, a quite standard
argument can be used to show that the density condition is actually
redundant.)

\subsection{The construction}

The construction is quite similar to other construction techniques
\cite{Kirby2009a,Mantova2011a}. We define the function $E$ by induction
using a back-and-forth procedure.

Let us fix $\omega\in\oQ^{*}$ and let us define our base function
as $E_{-1}(\frac{p}{q}\omega)=\zeta_{q}^{p}$, for $p,q\in\Z$, where
$\{\zeta_{q}\}_{q\in\Z}$ is a ``coherent'' system of roots of unity,
where by coherent we mean that $\zeta_{pq}^{p}=\zeta_{q}$ for all
$p,q\in\N$. This yields $\ker(E_{-1})=\omega\Z$.

Now let $\{\alpha_{n}\}$ be an enumeration of $\oQ^{*}$ and $V_{n}$
an enumeration of all the irreducible absolutely free algebraic varieties
$V_{n}$. At each step $n<\omega$ we proceed as follows:
\begin{enumerate}
\item If $\alpha_{n}$ is not in the domain of $E_{n-1}$, we choose some
$\beta\in\oQ^{*}\setminus\img(E_{n-1})$ and we define
\[
E_{n-1}^{1}(\alpha+\frac{p}{q}\alpha_{n}):=E_{n-1}(\alpha)\cdot\beta^{p/q},
\]
for all $\alpha\in\dom(E_{n-1})$ and $p,q\in\Z$, where $\beta^{1/q}$
is a coherent system of roots of $\beta$. If $\alpha_{n}$ is in
the domain, we just define $E_{n-1}^{1}:=E_{n-1}$.
\item If $\alpha_{n}$ is not in the image, we choose some $\beta\in\oQ\setminus\dom(E_{n-1}^{1})$
and we define
\[
E_{n-1}^{2}(\alpha+\frac{p}{q}\beta):=E_{n-1}^{2}(\alpha)\cdot\alpha_{n}^{p/q},
\]
for all $\alpha\in\dom(E_{n-1}^{1})$ and $p,q\in\Z$. If $\alpha_{n}$
is already in the image, we just define $E_{n-1}^{2}:=E_{n-1}^{1}$.
\item If $V_{n}\subset\oQ^{k}\times(\oQ^{*})^{k}$, we take a point $(\alpha_{1},\dots,\alpha_{k},\beta_{1},\dots,\beta_{k})\in V_{n}$
such that

\begin{enumerate}
\item $\alpha_{1},\dots,\alpha_{k}$ is $\Q$-linearly independent from
$\dom(E_{n-1}^{2})$;
\item $\beta_{1},\dots,\beta_{k}$ is multiplicatively independent from
$\img(E_{n-1}^{2})$;
\end{enumerate}

\noindent and we define $E_{n}(\alpha+\frac{p_{1}}{q_{1}}\alpha_{1}+\dots+\frac{p_{k}}{q_{k}}\alpha_{k}):=E_{n-1}^{2}(\alpha)\cdot\beta_{1}^{p_{1}/q_{1}}\cdot\dots\cdot\beta_{k}^{p_{k}/q_{k}}$
for all $\alpha\in\dom(E_{n-1}^{2})$ and $p_{i},q_{i}\in\Z$. 

\end{enumerate}
The limit function $E:=\bigcup_{n<\omega}E_{n}$ is the function we
sought in\ \prettyref{thm:main}. We can verify that the above construction
is sound; the only critical step is (3), since it is not completely
trivial that such a choice of $\alpha_{i}$ and $\beta_{i}$ is possible.
However, their existence can be deduced from \prettyref{prop:good-points},
which is described in the next section. 
\begin{proof}[Proof of \prettyref{thm:main}]
 It is immediate to see that the steps (1) and (2) can always be
performed, as $\dom(E_{n})$, $\img(E_{n})$, $\dom(E_{n}^{1})$ and
$\img(E_{n}^{1})$ are always finite-rank groups, and therefore we
can always find a suitable $\beta$.

Since $\dom(E_{n}^{2})$ and $\img(E_{n}^{2})$ are finite-rank subgroups
as well, \prettyref{prop:good-points} implies that $V_{n}$ contains
a point with the required properties.

It is again immediate to see that $E_{n}$ is a well defined function,
and in particular $E$ is well defined too, since $\dom(E_{n})$ is
always a $\Q$-vector space, and the new elements on which we define
the function are always $\Q$-linearly independent from the previous
domain. Moreover, $E$ is defined everywhere.

Similarly, $\ker(E)=\ker(E_{n})=\ker(E_{-1})=\omega\Z$, since every
time we define the new function, the new elements in the image are
multiplicatively independent from the previous image. Moreover, $E$
is surjective.

Finally, is is a standard argument to show that if every algebraic
variety $V$ contains a point of the form $(\overline{z},E(\overline{z}))$,
as it is the case for the function $E$ we constructed, then such
points are Zariski-dense.

Indeed, let $V$ be a given irreducible absolutely free algebraic
variety in $\oQ^{k}\times(\oQ^{*})^{k}$ and let $W\subset V$ be
a Zariski-closed proper subset of $V$. Without loss of generality,
we may assume that there is a polynomial $p\in K[x_{1},\dots,x_{2k}]$
such that $W=V\cap\{p=0\}$. It is now sufficient to consider the
variety $H\subset\oQ^{k+1}\times(\oQ^{*})^{k+1}$ defined by the same
equations defining $V$ on the first $k$ coordinates of $\oQ^{k+1}$
and of $(\oQ^{*})^{k}$, and by the equation $px_{2k+1}=1$, where
$x_{2k+1}$ is the last coordinate of $\oQ^{k+1}$. This variety must
contain a point of the form $(\overline{z}',E(\overline{z}'))$; its
projection on $\oQ^{k}\times(\oQ^{*})^{k}$ is a point of $V$ outside
of $W$. On varying $W$, this shows that such points are Zariski-dense
in $V$.
\end{proof}

\subsubsection*{Free exponential closure}

We want to remark the fact that our construction actually satisfies
the following ``free'' version of Exponential-algebraic closure:
\begin{enumerate}
\item [(FEC)] \emph{Free Exponential-algebraic Closure}: for every irreducible
absolutely free algebraic variety $V\subset\oQ^{n}\times(\oQ^{*})^{n}$,
and every finite subset $\overline{c}\subset\oQ$ such that $V$ is
defined over $\overline{c}$, there is a $\overline{z}\in K^{n}$
such that $(\overline{z},E(\overline{z}))\in V$ and $\overline{z}$
is $\Q$-linearly independent from $\overline{c}$.
\end{enumerate}
This behaviour mimics the genericity assumption in (SEC), and it is
in fact deeply related to it; in fact, it is shown in \cite{Kirby2010a}
that (SP) and (FEC) taken together imply (SEC).

%% file: GoodPoints.tex
\section{Points with independent coordinates}

\label{sec:goodpoints}

In order to finish the proof, we need to verify that absolutely free
algebraic variety always contain the points needed for step (3).

It is known that if we take a variety $V$ and some functions on it
that are multiplicatively independent (the functions are allowed to
be constant), then for ``most'' points $P\in V(\oQ)$ the values
of the functions at $P$ are still multiplicatively independent \cite{Masser1989}.

Similarly, it is also not difficult to show that for ``most'' points
the specialisations of $\Q$-linearly independent functions are still
$\Q$-linearly independent (again, the functions are allowed to be
constant). In order to put together the two statements, we first intersect
our variety with hyperplanes, using Bertini's theorem, to reduce to
the case when $V$ is a curve, and then we prove the case of curves.
We first take care of the additive part.
\begin{prop}
\label{prop:linear-good-points}Let $\Cc$ be an absolutely irreducible
curve defined over a field $K$, and let $k=\oQ\cap K$. Let $z_{1},\dots,z_{n}$
be some $\Q$-linearly independent functions in $K(\Cc)$. Let $x\in K(\Cc)$
be a non constant function.

There is a number $d>0$, not dependent on $z$, such that for any
$\alpha\in\oQ$ with $[k(\alpha):k]>d$, the specialisations of $z_{1},\dots,z_{n}$
at any non-singular point $P\in x^{-1}(\alpha)$ are $\Q$-linearly
independent.\end{prop}
\begin{proof}
Let $e$ be the maximum of $[K(\Cc):K(z_{i})]$ as $z_{i}$ ranges
among the non-constant functions.

Clearly, the equation 
\[
m_{1}x_{1}+\dots+m_{n}z_{n}=0,
\]
with the $m_{i}$'s not all zero, can be solved only in at most $ne$
points algebraic over $K$, since the function on the left is either
constant, hence non-zero by assumption, or it has degree at most $ne$.
We claim that that for any $\alpha\in\oQ$, if $[K(\alpha):K]=[k(\alpha):k]>ne$,
then any non-singular $P\in x^{-1}(\alpha)$ is such that $z_{1}(P),\dots,z_{n}(P)$
are $\Q$-linearly independent (note that $z_{1}$, $\dots$, $z_{n}$
have no zeroes or poles at $P$).

Indeed, let $\alpha$ and $P$ be given as above, and let $L$ be
a normal extension of $K$ that defines $P$. Clearly, $L\cap\oQ\supset k(\alpha)$
is a normal extension of $k$ by the assumption $k=K\cap\oQ$. Since
$\Cc$ is absolutely irreducible, we can extend the Galois action
of $\Gal(L/K)$ to $\Gal(L(\Cc)/K(\Cc))$. If there are $m_{1},\dots,m_{n}$
such that the above equation is satisfied, then by conjugation we
obtain several other $\sigma(P)$ satisfying the same equation. Since
$x(\sigma(P))=\sigma(\alpha)$, and $[k(\alpha):k]>ne$, we find more
than $ne$ distinct conjugates of $P$ all satisfying the above equation,
a contradiction.\end{proof}
\begin{cor}
\label{cor:linear-good-points}Let $\Cc$ be an absolutely irreducible
curve defined over $k$. Let $z_{1},\dots,z_{n}$ be some $\Q$-linearly
independent functions in $k(\Cc)$.

There is a number $d'>0$ such that for any $P\in\Cc(\overline{k})$
with $[k(P):k]>d'$, the specialisations of $z_{1},\dots,z_{n}$ at
$P$ are $\Q$-linearly independent.\end{cor}
\begin{proof}
Let us take a non-constant function $x\in k(\Cc)$ and let $e$ be
its degree.

Let $d$ be the number obtained by \prettyref{prop:linear-good-points}
applied to $x$ and $z_{1}$, $\dots$, $z_{n}$, and let $d'\geq d\cdot e$.
We take $d'$ large enough such that $P$ is non-singular and $x(P)$
is defined for each point with $[k(P):k]>d'$.

Now, if $P$ is a point such that $[k(P):k]>d':=d\cdot e$, then $x(P)$
is defined, finite and $[k(x(P)):k]>d$. By the previous proposition,
the specialisations of $z_{1},\dots,z_{n}$ at $P$ are $\Q$-linearly
independent.
\end{proof}
An analogous but different statement holds for the multiplicative
case for varieties of dimension greater than $1$.
\begin{prop}
\label{prop:mult-good-points}Let $V$ be an absolutely irreducible
variety defined over $k$ with $\dim(V)>1$. Let $w_{1},\dots,w_{n}$
be some functions in $k(V)$ that are multiplicatively independent
over $\overline{k}^{*}$.

There is a non-constant function $x\in k(V)$ such that the restrictions
of $w_{1},\dots,w_{n}$ at $V\cap x^{-1}(\alpha)$ are multiplicatively
independent over $\overline{k}^{*}$ for almost all $\alpha\in\overline{k}$.\end{prop}
\begin{proof}
Up to birational equivalence, we may assume that $V$ is smooth and
projective.

Since $w_{1},\dots,w_{n}$ are multiplicatively independent modulo
constants, it means that the Weil divisors of $w_{1},\dots,w_{n}$
are $\Q$-linearly independent. Up to taking a multiplicative combination
of the $w_{i}$'s, we may assume that there are $W_{1},\dots,W_{n}$
distinct prime divisors such that $w_{i}$ has a pole in $W_{i}$,
but has no zeroes and poles among the remaining $W_{j}$'s; in other
words, the matrix $(o_{W_{i}}(w_{j}))_{i,j}$ is diagonal, where $o_{W_{i}}(w_{j})$
is the order of $w_{j}$ at $W_{i}$.

Up to enlarging $k$, we may assume that these prime divisors have
degree $1$ and are all defined over $k$. It is clear that among
all the hyperplanes $H$ that intersect $V$ properly, the ones such
that $H\cap W_{i}=H\cap W_{j}$, with $i\neq j$, form a proper Zariski-closed
subset. By Bertini's theorem, it is also true that the ones such that
$H\cap V$ is not absolutely irreducible, and similarly the ones such
that $H\cap W_{i}$ is not absolutely irreducible, form proper Zariski-closed
sets.

Therefore, we can find an hyperplane $H$ represented by the equation
$x=0$ such that $x^{-1}(\alpha)\cap W_{i}$ and $x^{-1}(\alpha)\cap V$
are all smooth and distinct absolutely irreducible varieties for almost
all $\alpha\in\overline{k}$. But then the restrictions of $w_{1},\dots,w_{n}$
to $x^{-1}(\alpha)\cap V$ are such that $(o_{H\cap W_{i}}(w_{j}))_{i,j}$
is still a diagonal matrix, which implies that their divisors are
still $\Q$-linearly independent, hence the restrictions are multiplicatively
independent over $\overline{k}^{*}$.
\end{proof}
We shall use the above statements to reduce to the case of curves.
For curves, we adopt a different strategy.
\begin{prop}
\label{prop:curves-good-points}Let $\Cc\subset\oQ^{n}\times(\oQ^{*})^{n}$
be an irreducible curve over $\oQ$ and $L<\oQ$, $M<\oQ^{*}$ be
two finite-rank subgroups. If $\Cc$ is absolutely multiplicatively
free, and additively free over $L$, then there is a point $(\alpha_{1},\dots,\alpha_{n},\beta_{1},\dots,\beta_{n})\in\Cc$
such that $\alpha_{1},\dots,\alpha_{n}$ are $\Q$-linearly independent
from $L$, and such that $\beta_{1},\dots,\beta_{n}$ are multiplicatively
independent from $M$.\end{prop}
\begin{proof}
Without loss of generality, we may assume that $\Cc$ is absolutely
irreducible. Let $z_{1}$, $\dots$, $z_{n}$ and $w_{1}$, $\dots$,
$w_{n}$ be the coordinate functions of $\oQ^{n}\times(\oQ^{*})^{n}$
restricted to $\Cc$ and $a_{1},\dots,a_{m}$ be a finite set of divisible
generators of $M$. Let $k$ be a number field defining $\Cc$ and
containing $a_{1},\dots,a_{m}$.

Using the notation of \cite{Masser1989}, we define
\begin{itemize}
\item $\Cc(d,h)$ the set of all points of $\Cc$ of degree at most $d$
and height at most $h$;
\item $\mathcal{E}(d,h)$ the set of all points of $\Cc$ of degree at most
$d$ and height at most $h$ such that the specialisations of $w_{1},\dots,w_{n}$
are multiplicatively dependent on $M$;
\item $\omega(S)$, for a finite set $S$, the minimum degree of an hypersurface
containing all the points of $S$.
\end{itemize}
Applying the main result of \cite[\S 5]{Masser1989} to $\G_{m}(k(\Cc))$
and to the group generated by $w_{1},\dots,w_{n},a_{1},\dots,a_{n}$,
we find a function $c_{1}(d)$ and a number $k$ such that $\omega(\mathcal{E}(d,h))\leq c_{1}(d)h^{k}$,
while we also find a $c_{2}$ such that $\omega(\Cc(d,h))\geq\exp(c_{2}(d)h)$
when $d$ is at least the degree of $\Cc$. %
\footnote{The statement of \cite{Masser1989} is actually that $\omega(\Cc(d,h))\geq\exp(ch)$
when $d=\deg(\Cc)$. However, the proof only requires that there is
a dominant map $\pi:\Cc\to\P^{m}$ of degree $d$ with $m=\dim\Cc$.
Such maps exist for example for any multiple of $\deg(\Cc)$, as we
can compose $\pi$ with dominant self maps of $\P^{m}$ which exist
for any positive degree.%
}

Now using \prettyref{cor:linear-good-points} on $\Cc$ and $L$ we
obtain a number $d_{1}$ such that when $[k(P):k]>d_{1}$ the specialisations
of $z_{1}$, $\dots$, $z_{n}$ at $P$ are $\Q$-linearly independent
from $L$. We may choose $d_{1}$ larger than the degree of $\Cc$.
Now let $d_{2}$, $h_{1}$, $h_{2}$ be numbers such that

\begin{eqnarray*}
 & \omega(\Cc(d_{2},h_{2}))\geq\exp(c_{2}(d_{2})h_{2})>\omega(\Cc(d_{1},h_{1}))+c_{1}(d_{2})h_{2}^{k}\geq\\
 & \geq\omega(\Cc(d_{1},h_{1}))+\omega(\mathcal{E}(d_{2},h_{2}));
\end{eqnarray*}

Then there must be a point $P$ of degree strictly greater than $d_{1}$
such that the specialisations of $w_{1},\dots,w_{n}$ at $P$ are
multiplicatively independent from $a_{1},\dots,a_{n}$, hence from
$M$. Since its degree is greater than $d_{1}$, the specialisations
of $z_{1}$, $\dots$, $z_{n}$ are also $\Q$-linearly independent
from $L$, as desired.
\end{proof}
Putting the statements together, we can prove the general version
we need for step (3).
\begin{prop}
\label{prop:good-points}Let $V\subset\oQ^{n}\times(\oQ^{*})^{n}$
be an irreducible absolutely free variety over $\oQ$, and let $L<\oQ$,
$M<\oQ^{*}$ be two finite-rank subgroups. There is a point $(\alpha_{1},\dots,\alpha_{n},\beta_{1},\dots,\beta_{n})\in V$
such that $\alpha_{1},\dots,\alpha_{n}$ are $\Q$-linearly independent
from $L$ and $\beta_{1},\dots,\beta_{n}$ are multiplicatively independent
from $M$.\end{prop}
\begin{proof}
We prove the theorem by induction on $m=\dim(V)$. Our inductive hypothesis
is that if $V$ is absolutely irreducible, absolutely multiplicatively
free, and additively free over $L$, then it contains a point as in
the conclusion. The base case $m=1$ is covered by \prettyref{prop:curves-good-points}.
Let $k$ be a number field defining $V$.

Let us suppose that $m>1$, and that we have proven the theorem for
all the varieties of dimension $m-1$. Let $z_{1},\dots,z_{n}$ and
$w_{1}$, $\dots$, $w_{n}$ be the coordinate functions of $\oQ^{n}\times(\oQ^{*})^{n}$
restricted to $V$. Moreover, let $\{b_{1},\dots,b_{m}\}$ be a $\Q$-basis
of the vector space generated by $L$. By \prettyref{prop:mult-good-points},
there is a non-constant function $x$ such that for almost all $\alpha\in\overline{k}$
we have
\begin{enumerate}
\item $V_{\alpha}:=V\cap x^{-1}(\alpha)$ is absolutely irreducible;
\item $\dim(V_{\alpha})=m-1$;
\item the functions $\{w_{1},\dots,w_{n}\}$ restricted to $V_{\alpha}$
are multiplicatively independent over $\oQ^{*}$.
\end{enumerate}
Now take any transcendence base of $k(V)$ of the form $X\cup\{x\}$.
Then $V$ can be seen also as an absolutely irreducible curve over
$k(X)$, and $x$ is a non-constant function on it.

By applying \prettyref{prop:linear-good-points} to $V$ seen as a
curve over $K:=k(X)$, as soon as $[k(\alpha):k]$ is sufficiently
large, the functions $\{z_{1},\dots,z_{n},b_{1},\dots,b_{m}\}$ are
$\Q$-linearly independent when restricted to $V_{\alpha}$. Therefore
$V_{\alpha}$ satisfies the same properties of $V$, and by inductive
hypothesis, it contains a point $P$ such that the specialisations
of $z_{1}$, $\dots$, $z_{1}$ at $P$ are $\Q$-linearly independent
from $L$ and the specialisations of $w_{1}$, $\dots$, $w_{n}$
at $P$ are multiplicatively independent from $M$, as desired.\end{proof}
\begin{rem}
The above proof relies on the results exposed in \cite{Masser1989}.
These results depend on the Northcott Property of number fields. Using
other techniques of Diophantine geometry it is possible to obtain
a similar result for other finitely generated fields without the same
quantitative statements, but still strong enough to obtain again \prettyref{prop:curves-good-points}.
This implies that this construction works also on all algebraically
closed fields of characteristic $0$, and in particular of any fixed
transcendence degree.\end{rem}